\documentclass{amsart}
\usepackage{graphicx}
\usepackage{url}

\newtheorem{theorem}{Theorem}[section]
\newtheorem{lemma}[theorem]{Lemma}
\newtheorem{prop}[theorem]{Proposition}
\newtheorem{cor}[theorem]{Corollary}
\newtheorem{conj}[theorem]{Conjecture}

\theoremstyle{definition}
\newtheorem{definition}[theorem]{Definition}

\newtheorem*{notation*}{Notation}

\theoremstyle{remark}
\newtheorem{remark}[theorem]{Remark}

\numberwithin{equation}{section}

\newcommand{\abs}[1]{\lvert#1\rvert}

\usepackage{cite}
\nocite*{}

\title{cable knots do not admit cosmetic surgeries}

\author{Ran Tao}
\address{Department of Mathematics, The Chinese University of Hong Kong, Shatin, NT, Hong Kong}
\email{rtao@math.cuhk.edu.hk}

\begin{document}

\begin{abstract}
Two Dehn surgeries on a knot are called purely cosmetic if their surgered manifolds are homeomorphic as oriented manifolds. Gordon conjectured  that non-trivial knots in $S^3$ do not admit purely cosmetic surgeries. In this article, we confirm this conjecture for cable knots.

\end{abstract}

\maketitle

\section{introduction}

Let $K$ be a knot in a $3$-manifold $Y$, and let $r,s$ be two different numbers  in $ \mathbb{Q}\cup\{\infty\}$. By doing the $r$-slope Dehn surgery along $K$, we obtain a manifold $Y_r(K)$.   If $Y_r(K) \cong Y_{s}(K)$ as oriented manifolds, then the surgeries $r$ and $s$ are called \emph{purely cosmetic}; If $Y_r(K) \cong -Y_{s}(K)$ as oriented manifolds, then $r$ and $s$ are called \emph{chirally cosmetic}. Here, the manifold $-Y_{s}(K)$ denotes an oppositely-oriented copy of $Y_s(K)$.

There are many examples for chirally cosmetic surgeries of knots in $S^3$.  In fact, for each amphichiral knot $K$ and slope $r$, we have $S^3_{r}(K) \cong - S^3_{-r}(K)$. On the other hand, the right-handed trefoil provides a chiral knot example of chirally cosmetic surgeries.(See \cite{mathieu1992closed}.) However, no purely cosmetic surgeries are known until now. Actually, we have the following conjecture for purely cosmetic surgeries.

\begin{conj}[Cosmetic surgery conjecture\cite{gordon1990dehn},\cite{kirby1995problems}]

Suppose K is a knot in a closed oriented $3$-manifold $Y$ such that $Y\backslash K$ is irreducible and not homeomorphic to the solid torus. If two different Dehn surgeries on $K$ are purely cosmetic, then there is a homeomorphism of $Y\backslash K$ which takes one slope to the other.
\end{conj}

This conjecture in $Y = S^3$ case is already interesting enough and has been studied by many authors.  The knot complement theorem by Gordon and Luecke \cite{gordon1989knots} implies that the $\infty$-surgery can not be purely cosmetic with any other surgery slope.  Boyer and Lines \cite{boyer1990surgery} proved that if a knot $K$ admits purely cosmetic surgeries, then the normalized Alexander polynomial $\Delta_K(t)$ of $K$ satisfies $\Delta_K''(1) = 0$. In addition, Ichihara and Wu \cite{ichihara2016note} proved that the Jones polynomial $V_K(t)$ of $K$ satisfies $V'''_K(1) = 0$ in this case. In \cite{2016arXiv160202371I}, Ichihara and Saito obtained some partial results for two-bridge knots using the $SL(2,\mathbb{C})$ Casson invariant. In \cite{ito2017lmo}, Ito gave some constraints using the LMO invariants.  Besides the above criteria, Heegaard-Floer homology is also a powerful tool in studying this conjecture. Using the Heegaard-Floer correction term, Wang proved  in \cite{wang2006cosmetic} that genus one knots do not admit purely cosmetic surgeries.
 In \cite{ozsvath2010knot}, Ozsv\'{a}th and Szab\'{o} proved that if $S^3_r(K) \cong S^3_s(K)$, then either  $S^3_r(K)$ is an $L$-space or $rs < 0$.  Wu further \cite{wu2011cosmetic} proved that the first case can not happen, thus we must have $rs<0.$ Later in \cite{ni2015cosmetic}, Ni and Wu proved the following theorem, which puts strong restrictions on purely cosmetic surgeries.

\begin{theorem}[\cite{ni2015cosmetic}]
\label{ni-wu}
 Suppose $K$ is a non-trivial knot in $S^3$. If $S^3_{r}(K) \cong S^3_{s}(K)$ with $r \neq s$ as oriented manifolds, then we have the following: 
 \begin{itemize}
 \item $r = - s$.
 \item $n^2 \equiv -1  \pmod  m$. Here $r = m/n$, and $m, n$ are coprime.
 \item $\tau(K) = 0$. Here $\tau$ is the concordance invariant defined by Ozsv\'{a}th-Szab\'o and Rasmussen. (See \cite{ozsvath2003knot} and \cite{MR2704683} )

 \end{itemize}
 \end{theorem}

In this article, we study the purely cosmetic surgeries of \emph{cable knots}. A \emph{cable knot} $C_{p,q}(K)$ is the image of the torus knot $T_{q,p}$ lying on the boundary torus of a tubular neighborhood of $K$. We also regard $C_{p,q}(K)$ as a curve lying inside the tubular neighborhood by pushing the torus knot  $T_{q,p}$ into this solid torus. This curve winds $q$ times along the longitudinal direction of $K$ and $p$ times along the meridional direction. The main result is the following theorem. 

\begin{theorem}
\label{main}
Let $C_{p,q}(K)$ be a cable knot with winding number $\abs{q} \ge 2$. Suppose there exists an  orientation-preserving homeomorphism $f: S^3_{r}(C_{p,q}(K)) \to S^3_{s}(C_{p,q}(K)),$ then $r = s$.

\end{theorem}

Note that we only need to compare $S^3_{r}(C_{p,q}(K))$ and $S^3_{-r}(C_{p,q}(K))$ by Theorem \ref{ni-wu}. Furthermore, we can reduce this problem to the case $p = \pm 1$. The idea of this reduction is that the concordance invariant $\tau(C_{p,q}(K))$  vanishes if and only if $p = \pm 1$. This observation is based on the formulae for $\tau(C_{p,q}(K))$ computed by Jennifer Hom in \cite{hom2013bordered}. (The notation for $C_{p,q}(K)$ in \cite{hom2013bordered} is different from ours.) We do not need this reduction in the proof of Theorem \ref{main}. \\

 This result is new in the sense that Theorem \ref{main} does not follow  from the criteria mentioned above. In fact, we have the following: 

 \begin{prop}
 If $K$ is a non-trivial knot such that $\Delta''_K(1) = 0$ and $V'''_K(1) = 0$, then we have $\Delta''_{C_{\pm 1,q}(K)}(1) = 0$ and $V'''_{C_{\pm 1,q}(K)}(1) = 0$.
 \end{prop}

 \begin{proof}

We need some basic facts of finite type invariants of knots.  The set $U_k$ of all finite type invariants of order $\leq n$ can be regarded as a vector space. When $n = 0$ or $1$, $U_n$ consists exactly of the constant fuctions. We also have $\dim(U_2) = 2$ and $\dim(U_3) = 3.$ 

Recall that $\Delta''_K(1)$ is a finite type invariant of order $2$ and $V'''_K(1)$ is of order $3$. Then each finite type invariant of order $2$ or $3$ is a linear combination of the constant function $1$, $\Delta''_K(1)$, and $V'''_K(1)$. (See \cite{bar1995vassiliev}; see also \cite{ichihara2016note}.)  Also recall that finite type invariants have a cabling property, which states that if $u_n$ is a finite type invariant of order $n$, then $u_n(C_{p,q}(K))$ is also a finite type invariant of order $n$. (See Proposition 9.9 of \cite{chmutov2012introduction}.)

In our case, we have $\Delta''_{C_{\pm 1,q}(K)}(1) = a\cdot \Delta''_{K}(1) + b$ and $V'''_{C_{\pm 1,q}(K)}(1) = c\cdot V'''_{K}(1) + d \cdot \Delta''_{K}(1) + e $, since $\Delta''_{C_{\pm 1,q}(K)}(1)$ and $V'''_{C_{\pm 1,q}(K)}(1)$ are finite type invariants of order $2$ and $3$ respectively.  Suppose $K$ is the unknot, then $C_{\pm 1,q}(K)$ is also the unknot. Since all the finite type invariants of the unknot is zero, we have $b = e = 0$. Hence if $\Delta''_{K}(1) = 0$, then $\Delta''_{C_{\pm 1,q}(K)}(1) = 0$; if we also have $V'''_{K}(1) = 0$, then $V'''_{C_{\pm 1,q}(K)}(1) = 0.$ 
 \end{proof}

Our proof relies on Gordon's classification of Dehn surgeries on cable knots.

\begin{theorem}[Corollary 7.3 of \cite{gordon1983dehn}]
\label{classification}
Suppose $q \ge 2$. The surgered manifold $S^3_{r}(C_{p,q}(K))$ of a cable knot $C_{p,q}(K)$ with slope $r = m/n$ is classified as follows.
\[
S^3_{r}(C_{p,q}(K)) = 
\begin{cases}
 S^3_{p/q}(K) \# L(q,p), &\text{if}  ~r = pq,\\
 S^3_{r/q^2}(K),  &\text{if} ~ m = npq \pm 1,\\
 ( S^3 \backslash N(K) ) \cup_{T} SFS_r,  &otherwise.\\  
\end{cases}
\]
Here, the notation $SFS_r$ denotes a Seifert fibered space with two singular fibers of multiplicities $\abs{q} $ and $\abs{npq - m}$. The boundary torus of $SFS_r$ is incompressible.

\end{theorem}


To prove Theorem \ref{main}, we simply compare $S^3_r(C_{p,q}(K))$ and $S^3_{-r}(C_{p,q}(K))$ case by case, as described in Theorem \ref{classification}. The outline is the following: if $S^3_r(C_{p,q}(K))$ appears in the first case, then $S^3_{-r}(C_{p,q}(K))$ must be irreducible; if $S^3_r(C_{p,q}(K))$ appears in the second case, then $S^3_{-r}(C_{p,q}(K))$ has one more JSJ-torus than that of $S^3_{r}(C_{p,q}(K))$; if $S^3_r(C_{p,q}(K))$ appears in the last case, then the JSJ-pieces of $S^3_r(C_{p,q}(K))$ and $S^3_{-r}(C_{p,q}(K))$ are different. In each case, we have $S^3_r(C_{p,q}(K))\not\cong S^3_{-r}(C_{p,q}(K)),$ and hence the theorem is proved.\\

\subsection*{Acknowledgements.} I wish to thank my advisor Zhongtao Wu for his help and guidance throughout my Ph.D. studies. He introduced me the cosmetic surgery conjecture and gave me many helpful suggestions on this problem. I also thank Kazuhiro Ichihara, Kimihiko Motegi and Jingling Yang for helpful comments and discussions. This work was partially supported by grants from the Research Grants Council of the Hong Kong Special Administrative Region, China (Project No. CUHK 14301215).


\section{The JSJ-decompositions of knot complements}

\begin{notation*}  The notation $M \cong M'$ means that there exists an orientation-preserving homeomorphism between oriented $3$-manifolds $M$ and $M'$. All the $3$-manifolds in this article are assumed to be compact and oriented. 

We use $N(K)$ to denote a tubular neighborhood of the knot $K$ in a $3$-manifold. For a knot $K$ in $S^3$,  we use $E(K)$ to denote $S^3 \backslash {N(K)}$. For each knot $K$ in $S^3$, a prefered longitude is chosen to be homologically trivial in $E(K)$. A $(p,q)$-curve or a $p/q$-slope on the boundary of $N(K)$ or $E(K)$ is a curve that winds $p$ times along the meridional direction and $q$ times along a prefered longitude. Note that we can also talk about the $(p,q)$-curves inside $N(K)$ once the prefered longitude is chosen.   We use $C_{p,q}(K)$ to denote the $(p,q)$-cable of $K$, with longitudinal winding number $\abs{q}\ge 1$. Since $C_{p,\pm 1} (K)$ is isotopic to $K,$ we require that $\abs{q} \ge 2$.  Given two slopes $r$ and $s$, we use $\Delta(r,s)$ to denote their minimal geometric intersection number on the torus.

Let $M$ be a $3$-manifold with a toroidal boundary component. Suppose we have chosen meridians and longitudes on this boundary torus. Then we use $M(r)$ to denote the resulting manifold of the $r$-slope Dehn filling on $M$ along this boundary.  

\end{notation*}

In this section, we collect some results related to JSJ-decomposition. First we recall the JSJ-decomposition Theorem, which is due to Jaco and Shalen \cite{MR520524}  and Johannson \cite{MR551744}. We use the version in Hatcher's notes \cite{hatcher2000notes}. 

\begin{theorem}[The JSJ-decomposition theorem]
Let $M$ be a compact irreducible orientable $3$-manifold. Then there exists a finite collection of disjoint incompressible tori $\{T_i\}$ in $M$ such that each component of $M\backslash \cup_i T_i$ is either atoroidal or Seifert fibered. Furthermore, a minimal choice of such a collection  is unique up to isotopy. 
\end{theorem}

\begin{definition}
We call the unique isotopy class of decomposition tori (or any representative) in the above theorem the \emph{JSJ-tori} of $M$. We call an embedded torus $T$ a \emph{JSJ-torus} if $T$ is isotopic to a torus in the collection of {JSJ-tori}. We also call the components resulting from decomposing $M$ along the {the JSJ-tori}  \emph{the JSJ-pieces} of $M$. We just call an object \emph{JSJ} for short in the above cases if there is no ambiguity.
\end{definition}

\begin{remark}

We can apply the JSJ-decomposition theorem to a manifold with incompressible toroidal boundary by considering its double. In particular, knot complements admit JSJ-decompositions. See \cite{budney2005jsj} for an explicit description of this JSJ-structure.

\end{remark}

Let $f: M \to M'$ be a homeomorphism of oriented $3$-manifolds. Then $f$ maps the JSJ-tori (or JSJ-pieces) of $M$ to the JSJ-tori (or JSJ-pieces) of $M'$ bijectively. In other words, $M$ and $M'$ have identical sets of JSJ-pieces. The following lemma is the key idea of the proof of Theorem \ref{main}.

\begin{lemma}
\label{diff_sfs_piece}
Let $N$, $F$, and $F'$ be compact irreducible $3$-manifolds with toroidal boundaries. Suppose $F$ and $F'$ are atoroidal or Seifert fibered. Let $M = N \cup_T F$ and $M' = N \cup_{T'} F'$ be manifolds obtained by gluing along the boundary tori. Suppose further that the gluing tori $T$ and $T'$ are JSJ in $M$ and $M'$. If $F \not\cong F'$, then $M \not\cong M'.$ 
\end{lemma}

\begin{proof}
Suppose we have an orientation perserving map  $f: M \to M'$. Then $f$ induces a bijection between the two sets of JSJ-pieces, sending each piece to its homeomorphic image in the other set. Each set  consists of the JSJ-pieces of $N$ and an extra piece $F$ or $F'$.  Since $F \not\cong F'$, the two sets can not be equal since they are finite. Indeed, if $F$ is mapped into $N$, then $N$ contains a JSJ-piece which is homeomorphic to $F$. Similarly, $N$ should also contain a JSJ-piece which is homeomorphic to $F'$. By induction, we can deduce that $N$ contains infinitely many $F$ and $F'$, which is impossible.
\end{proof}

We need a criterion on whether certain tori are JSJ.

\begin{prop}[Proposition 1.6.2 of \cite{aschenbrenner3}]
\label{criterion}

Let $M$ be a compact irreducible orientable $3$-manifold with empty or toroidal boundary. Let $\{T_i\}$ be a collection of disjoint embedded incompressible tori in $M$. Then $\{T_i\}$ are the JSJ-tori of $M$ if and only if the following holds: 
\begin{enumerate}

\item each component $\{M_j\}$ of $M\backslash \cup_i T_i$ is atoroidal or Seifert fibered;

\item if  $~T_i$ cobounds Seifert fibered components $M_j$ and $M_k$ (with possibly $j = k$), then their regular fibers do not match; in other words, their Seifert fibered structures can not be glued together along $T_i$ to form a larger one;

\item if a component $M_i$ is homeomorphic to $T^2 \times I$, then $M$ is a torus bundle with only one JSJ-piece.
\end{enumerate}

\end{prop}

\begin{remark}
\label{gluing_torus_is_JSJ_remark}
We can use the above criterion to determine whether certain gluing tori are JSJ. Let $M = M_1 \cup_T M_2$ be a manifold obtained by gluing $M_1$ and $M_2$ along a toroidal boundary. Suppose $M_1$ does not have torus bundle JSJ-pieces and $M_2$ is Seifert fibered. Suppose further $T$ is incompressible in both $M_1$ and $M_2$. If the JSJ-piece bounded by $T$ in $M_1$ is Seifert fibered, and its regular fibers on $T$ are glued onto the regular fibers of $M_2$, then  $T$ is not a JSJ-torus of $M$. The above criterion tells us that this is the only case where $T$ is not JSJ.

\end{remark}

We need a description of the JSJ-decomposition of knot complements in $S^3$. The following theorem is due to Budney \cite{budney2005jsj}, and is based on previous works by Jaco and Shalen \cite{MR520524}, Johannson \cite{MR551744}, Bonahon and Siebenmann \cite{bonahonsiebenmann_unpublish}, Eisenbud and Neumann \cite{MR817982}, and Thurston \cite{MR648524}. We use the version reformulated by Lackenby \cite{lackenby2017every}.

\begin{theorem}[Theorem 4.1 of \cite{lackenby2017every}, Theorem 4.18 of \cite{budney2005jsj}]
\label{jsj_of_knot_complement}

Suppose $K$ is a knot in $S^3$ such that $E(K)$ has at least one JSJ-torus. (In this case, $K$ is a satellite knot).  Let $M$ be the JSJ-piece of $E(K)$ containing $\partial N(K)$. Then $M$ has one of the following forms:

\begin{enumerate}

\item an annulus based Seifert fibered space with one singular fiber; in this case, the knot $K$ is  a cable of a non-trivial knot.

\item a planar surface based Seifert fibered space with at least three boundary components and with no singular fibers; each regular fiber has the meridional slope on $\partial N(K)$; 


\item a hyperbolic manifold which is homeomorphic to the complement of some hyperbolic link $L$ in $S^3$; this link becomes a trivial link or the unknot if the component corresponding to $\partial N(K)$ is removed; the homeomorphism from the JSJ-piece $M$ to $S^3 \backslash N(L)$ sends each slope $p/q$ on $\partial N(K) $ to the slope $p/q$ on the corresponding component of $\partial N(L)$;


\end{enumerate}

\end{theorem}

The JSJ-piece $M$ in the first case is called a \emph{cable space}. In this case, the knot $K$ is  a cable knot $C_{p',q'}(K')$. By the definition of cable knots, we can regard $K$ as a $(p',q')$-curve in $N(K') \subset S^3$.  Then the JSJ-piece $M$ is the closure of $N(K') \backslash N(K)$, where $N(K)$ is chosen to be a union of $(p',q')$-curves in $N(K')$. The slope of each regular fiber of such a cable space on $\partial N(K)$ is described in the following lemma.

\begin{lemma}[cf. Lemma 7.2 of \cite{gordon1983dehn}]
\label{cable_space_fiber_slope}
Suppose $K$ is a cable knot $C_{p',q'}(K')$, then the closure of  $N(K') \backslash N(K)$ is a Seifert fibered space. Each of its regular fibers  has slope $p'/q'$ on $\partial N(K')$ and  has slope $p'q'/1$ on $\partial N(K)$.

\end{lemma}

\begin{proof}
The regular fibers of the Seifert fibered solid torus $N(K')$ are defined to be the $(p',q')$-curves. In particular, each regular fiber has slope $p'/q'$ on $\partial N(K')$. The Seifert fibered space $N(K')$ has a singular fiber of multiplicity $\abs{q'}$, which is the core of $N(K')$. The knot $K$ can be identified as a regular fiber of $N(K')$. Now $N(K') \backslash N(K)$ inherits the Seifert fibered structure from $N(K')$. Given two regular fibers of $N(K')$, we see that their linking number is $\pm p'q'$, and hence the regular fibers on $\partial N(K)$ are the $(p'q',1)$-curves.
\end{proof}

\section{JSJ-decompositions and Dehn surgeries on Cable knots }

In this section, we study the JSJ-structure of the surgered manifold $S^3_r(C_{p,q}(K))$. First, we need to know whether a surgery creates  essential tori in our case. The following theorems on toroidal surgeries by Gordon and Luecke are useful for us.

\begin{theorem}[\cite{gordon1995dehn}]
\label{toroidal_surgery_s3}
Suppose $K$ is a hyperbolic knot in $S^3$. If $S^3_{i/j}(K)$ contains an essential torus, then $\abs{j} \leq 2$.
\end{theorem}

\begin{theorem}[\cite{gordon1999toroidal}]
\label{toroidal_surgery}
Let $M$ be a simple $3$-manifold with a toroidal boundary component. We use $M(r)$ and $M(s)$ to denote its $r$-slope and  $s$-slope Dehn fillings along this boundary component. If $M(r)$ is boundary reduci{}ble and $M(s)$ is toroidal, then $\Delta(r,s) \leq 2.$
\end{theorem}

\begin{cor}[cf. Theorem 2.8 of \cite{lackenby2017every}]
\label{toroidal_surgery_not_solid_torus}
Let $M$ be the complement of an unlink or the unknot in $S^3$. Let $K$ be a knot in $M$ such that $M\backslash N(K)$ is hyperbolic. Then the $p/q$-slope Dehn surgery with $q > 2$ on $K$ is atoroidal and each boundary component of the surgered manifold is incompressible.  
\end{cor}

\begin{proof}
The manifold $M(\infty)$ is boundary reducible. By Theorem \ref{toroidal_surgery}, the $p/q$-slope surgery is atoroidal since $\Delta(\infty,p/q)  = \abs{q} \ge 2$. In this case, there is no essential disk in this surgered manifold, by \cite{MR1167169}.
\end{proof}

Now we  describe the JSJ-decomposition of  $S^3_r(C_{p,q}(K))$. By Theorem \ref{classification}, the surgered manifold is classified into three cases. Note that only the first case corresponds to reducible manifolds. (See Theorem \ref{cable_conjecture}.) The remaining two cases correspond to manifolds obtained by gluing the knot complement of $K$ in $S^3$ with a solid torus or with a Seifert fibered space.

In the following two lemmas, we assume that the knot complement $E(K)$ has at least one JSJ-torus. In this case, $K$ is a satellite knot. 
We start with the second case of Theorem \ref{classification}.

\begin{lemma}
\label{knot_complement_glue_torus}
Let $M$ be the JSJ-piece in $E(K)$ containing $\partial N(K)$. Then the Dehn filling $M(r/q^2)$ in   $S^3_{r/q^2}(K)$ is atoroidal or Seifert fibered, and the boundary tori are incompressible.  
\end{lemma}

\begin{proof}
Suppose $M$ is hyperbolic. By Theorem \ref{jsj_of_knot_complement}, the JSJ-piece $M$ is homeomorphic to $S^3 \backslash N(L)$, where $L$ is a hyperbolic link. The link $L$ has the property that, if we remove the component corresponding to $\partial N(K)$, then $L$ becomes the unknot or an unlink.  Note that $r/q^2 = m/nq^2 = (npq\pm 1)/nq^2$, and the denominator is coprime with the numerator.  Since the denominator satisfies $nq^2 > 2$, the Dehn filling $M(r/q^2)$ is atoroidal and its boundary tori are incompressible by Corollary \ref{toroidal_surgery_not_solid_torus}.


Suppose $M$ is a Seifert fibered space. We want to show that the Seifert fibered structure of $M$ extends to a Seifert fibered structure of $M(r/q^2)$. In other words, the meridian of the solid torus is not glued to a regular fiber of $M$. Thus we  need to compare the slopes. By Theorem \ref{jsj_of_knot_complement}, there are two possibilities when $M$ is Seifert fibered. First, suppose $M$ is a cable space. In this case, the knot $K$ is a cable knot $C_{p',q'}(K')$, by Theorem \ref{jsj_of_knot_complement}. Each regular fiber of $M$ has slope $p'q'/1$ on $\partial N(K)$, by Lemma \ref{cable_space_fiber_slope}. On the other hand, the surgery slope is $m/nq^2$. These two slopes do not match since $\abs{q} \ge 2.$ Now we suppose $M$ has a planar base. By Theorem \ref{jsj_of_knot_complement}, the slope of each regular fiber of $M$ on $\partial N(K)$ is meridional, which is $\infty$, and hence does not match the gluing slope $m/nq^2$.  

We also need to show that the boundary tori of the new Seifert fibered spaces $M(r/q^2)$ are incompressible. The case that $M$ is a cable space follows from Theorem \ref{classification}. Indeed, each Seifert fibered space $SFS_r$ in Theorem \ref{classification} is exactly a Dehn filling $M \cup_{T^2} (S^1 \times D^2)$ such that the Seifert fibered structure of the cable space $M$ extends. Now suppose $M$ is a planar based Seifert fibered spaces as in Theorem \ref{jsj_of_knot_complement}. Since $M(r/q^2)$ is Seifert fibered, each essential disk is either vertical or horizontal (Proposition 1.11 of \cite{hatcher2000notes}). In our case, an essential disk can not be a union of fibers, and hence not vertical. It can not be horizontal either, since a disk can not be a branched cover of a planar surface with more than two boundary components. Hence there are no essential disks in  $M(r/q^2)$.

\end{proof}
The following lemma corresponds to the third case of Theorem \ref{classification}.
\begin{lemma}
\label{knot_complement_glue_SFS}
Let $M$ be the JSJ-piece in $E(K)$ containing $\partial N(K)$. Then the gluing torus of $M \cup_{T^2} SFS_r$ is JSJ in $S^3_{r}(C_{p,q}(K))$.

\end{lemma}

\begin{proof}
The proof is similar to Lemma \ref{knot_complement_glue_torus}. First we determine the slope of each regular fiber of $SFS_r$ along the gluing torus $\partial N(K)$.  By the construction of Theorem \ref{classification},  the Seifert fiber space $SFS_r$ is obtained by doing the $r$-slope Dehn filling on the cable space $N(K) \backslash N( C_{p,q} (K)) $ along $\partial N( C_{p,q} (K))$. The Dehn filling does not affect the Seifert fibered structure near $\partial N(K)$. By Lemma \ref{cable_space_fiber_slope}, the  slope on  $\partial N(K)$ is $p/q$.

Suppose $M$ is hyperbolic. Then the gluing torus is JSJ of the surgered manifold by  Remark \ref{gluing_torus_is_JSJ_remark}.

Suppose $M$ is Seifert fibered. We claim that $M\cup_{T^2} SFS_r$ is not Seifert fibered. In other words, we want to show that the regular fibers of the two Seifert fibered pieces have different boundary slopes. In this case, the gluing torus is a JSJ-torus by Remark \ref{gluing_torus_is_JSJ_remark}. First, suppose $M$ is a cable space. As in the proof of Lemma \ref{knot_complement_glue_torus}, each regular fiber of $M$ has  slope $p'q'/1$ on $\partial N(K)$. Then $p'q'/1 \neq p/q$ since $\abs{q} \ge 2$.  Now we suppose $M$ is a planar based Seifert fibered space. Each regular fiber of $M$ has slope $\infty$ on $\partial N(K)$, by Theorem \ref{jsj_of_knot_complement}. We see that $\infty \neq p/q$ since $\abs{q}\ge 2$.

\end{proof}

The following two lemmas deal with the hyperbolic case and the torus knot case respectively. Their proofs are similar to the lemmas above.

\begin{lemma}
\label{hyperbolic_knot_case}
Suppose $K$ is a hyperbolic knot. Then $S^3_{r/q^2}(K)$ is atoroidal, while the gluing torus in $E(K) \cup_{T^2} SFS_r$ is JSJ. 
\end{lemma}

\begin{proof}
The proof of the first statement is similar to the first part of Lemma \ref{knot_complement_glue_torus}. We use theorem \ref{toroidal_surgery_s3} to conclude that $S^3_{r/q^2}( C_{p,q}(K))$ can not be toroidal since $\abs{nq^2} > 2$. The second statement follows from  Remark \ref{gluing_torus_is_JSJ_remark}.

\end{proof}

\begin{lemma}
\label{torus_knot_case}
Suppose $K$ is a torus knot. Then $S^3_{r/q^2}(K)$ is Seifert fibered, while the gluing torus in $E(K) \cup_{T^2} SFS_r$ is JSJ. 
\end{lemma}

\begin{proof}
The Dehn surgeries of torus knots are well-understood, see \cite{moser1971elementary}. In particular, $S^3_{r/q^2}( K)$ is a lens space or a Seifert fibered space with three singular fibers over $S^2$. The proof of the second statement is similar to Lemma \ref{knot_complement_glue_SFS}.

\end{proof}

\section{Proof of the main theorem}


We prove Theorem \ref{main} in this section. We compare $S^3_r(C_{p,q}{(K)})$ and $S^3_{-r}(C_{p,q}{(K)})$ case by case based on the classification of Theorem \ref{classification}. When $K$ is the unknot, we have the following well-known result.

\begin{theorem}
\label{cosmetic_torus_knot}
Torus knots  do not admit purely cosmetic surgeries.

\end{theorem}

\begin{proof}
 Let $K$ be a torus knot. We only need to show that $\Delta''_K(1) \neq 0.$  (See \cite{boyer1990surgery} \cite{ni2015cosmetic}, \cite{ichihara2016note}.) For non-trivial torus knots, we have  $\Delta_{T_{p,q}}''(1) = \frac{(p^2 - 1)(q^2 - 1)}{12} \neq 0.$
\end{proof}

We need the following theorem to deal with the first case of Theorem \ref{classification}. 

\begin{theorem}[\cite{scharlemann1990producing}]
\label{cable_conjecture}
Let $K$ be a satellite knot. If $S^3_{r}(K)$ contains an essential sphere, then $K$ is a $(p,q)$-cable and $r = pq$.
\end{theorem}

By this theorem, we see that only the first case of Theorem \ref{classification} involves reducible manifolds. Now we compare $S^3_r(C_{p,q}(K))$ and $S^3_{-r}(C_{p,q}(K))$ in the following three lemmas, which complete the proof of our main theorem. We use the notations from Theorem \ref{classification}.

\begin{lemma}
\label{case_1_remove}

 If $S^3_{r}(C_{p,q}(K))$ is of the form $S^3_{p/q}(K) \# L(q,p)$, then  $S^3_{r}(C_{p,q}(K)) \not\cong S^3_{-r}(C_{p,q}(K))$.
 
\end{lemma}

\begin{proof}
 By Theorem \ref{classification}, we have $r = pq \neq -r$, and hence $S^3_{-r}(C_{p,q}(K))$ does not correspond to the first case. However, manifolds appearing in the remaining two cases of Theorem \ref{classification}  are irreducible, by Theorem \ref{cable_conjecture}. 
\end{proof}

\begin{lemma}
\label{case_2-3_remove}
If $S^3_{r}(C_{p,q}(K))$ is of the form $S^3_{r/q^2}(K)$, then $S^3_{-r}(C_{p,q}(K))$ is of the form $E(K) \cup_{T^2} SFS_{-r}$. Furthurmore, we have $S^3_{r}(C_{p,q}(K)) \not\cong S^3_{-r}(C_{p,q}(K))$.
\end{lemma}

\begin{proof}
We have $m = npq \pm 1 \neq -m$, hence the first statement follows from Theorem \ref{classification}. Since $S^3_{r}(C_{p,q}(K))$ is of the form $S^3_{r/q^2}(K)$, it has one fewer JSJ-torus than that of $S^3_{-r}(C_{p,q}(K))$, by Lemma \ref{knot_complement_glue_torus}, \ref{knot_complement_glue_SFS}, \ref{hyperbolic_knot_case}, and \ref{torus_knot_case}. Thus we have $S^3_{r}(C_{p,q}(K)) \not\cong S^3_{-r}(C_{p,q}(K))$.
\end{proof}

\begin{lemma}
\label{case_3-3_remove}
If both $S^3_{r}(C_{p,q}(K))$ and $S^3_{-r}(C_{p,q}(K))$ are of form $E(K) \cup_{T^2} SFS_{\pm r}$, then they have different JSJ-pieces and hence are not homeomorphic. 
\end{lemma}

\begin{proof}
By Lemma \ref{knot_complement_glue_torus}, \ref{knot_complement_glue_SFS}, \ref{hyperbolic_knot_case}, and \ref{torus_knot_case}, we know that these two manifolds have the same number of JSJ-tori, and the gluing torus $\partial N(K)$ is JSJ in each manifold. By Lemma \ref{diff_sfs_piece}, we only need to show that $SFS_r \not\cong SFS_{-r}$. Note that the Seifert fibered piece $SFS_r$ contains a singular fiber of multiplicity $\abs{npq -m}$, while $SFS_{-r}$ contains a singular fiber of multiplicity $\abs{npq +m}$. However, we have $\abs{npq -m} \neq \abs{npq +m}$ since $m$, $n$, $p$, and $q$ are all non-zero. Thus $SFS_r \not\cong SFS_{-r}$ and the statement is proved.
\end{proof}

\begin{theorem}

Let $C_{p,q}(K)$ be a cable knot. Suppose there exists an orientation-perserving homeomorphism $f: S^3_{r}(C_{p,q}(K)) \to S^3_{s}(C_{p,q}(K)).$ Then $r = s$.

\end{theorem}

\begin{proof}
The statement follows from Theorem \ref{cosmetic_torus_knot} if $K$ is the unknot. The remaining cases follow from Lemma \ref{case_1_remove}, \ref{case_2-3_remove},  \ref{case_3-3_remove}, based on Theorem \ref{classification}.
\end{proof}

\bibliographystyle{halpha_1}
\bibliography{cable_knot}

\end{document}